\documentclass[12pt,draft]{amsart}
\usepackage{amssymb,extsizes}

\usepackage{geometry}
\theoremstyle{plain}
\newtheorem{Theorem}{Theorem}
\newtheorem{Corollary}[Theorem]{Corollary}
\newtheorem{Proposition}[Theorem]{Proposition}
\newtheorem{Lemma}[Theorem]{Lemma}

\renewcommand\l{\lambda}
\renewcommand\S{S(H)}
\newcommand\tr{\operatorname{Tr}}
\newcommand\supp{\operatorname{supp}}
\newcommand\HS{\operatorname{HS}}
\newcommand\Diag{\operatorname{Diag}}

\def\R{\mbox{$\mathbb{R}$}}

\begin{document}

\title[]{Transformations on density operators and on positive definite operators preserving the quantum R\'enyi divergence}
\author{MARCELL GA\'AL AND LAJOS MOLN\'AR}
\address{Department of Analysis, Bolyai Institute\\
University of Szeged\\
H-6720 Szeged, Aradi v\'ertan\'uk tere 1.,
Hungary and
MTA-DE ``Lend\" ulet'' Functional Analysis Research Group, Institute of Mathematics\\
         University of Debrecen\\
         H-4010 Debrecen, P.O. Box 12, Hungary}
\email{molnarl@math.u-szeged.hu}
\urladdr{http://www.math.u-szeged.hu/\~{}molnarl/}  
\address{}
\email{marcell.gaal.91@gmail.com}

\thanks{The second author was supported by the "Lend\" ulet" Program (LP2012-46/2012) of the Hungarian Academy of Sciences and by the Hungarian Scientific Research Fund (OTKA) Reg. No. K115383}

\begin{abstract}
In a certain sense we generalize the recently introduced and extensively studied notion called quantum R\'enyi divergence (in another name, sandwiched R\'enyi relative entropy) and describe the structures of corresponding symmetries. More precisely, we characterize all transformations on the set of density operators which leave our new general quantity invariant and also determine the structure of all bijective transformations on the cone of positive definite operators which preserve the quantum R\'enyi divergence.
\end{abstract}

\maketitle

\section{Introduction and formulation of the results}

We begin with a brief survey of former results which have led us to the investigation of the problem described in the abstract. 
Relative entropy is one of the most important numerical quantities
in quantum information theory. It is used as a measure of distinguishability between quantum states, or their mathematical representatives, the density operators. In fact, there are several
concepts of relative entropy among which the most common one is due to Umegaki. In \cite{M} the second author determined the general form of all bijective transformations on the set of density operators which preserve that type of relative entropy. The motivation to explore the structure of those transformations came from the fundamental theorem
of Wigner concerning quantum mechanical symmetry transformations. Those transformations are bijective maps on the set of pure states (whose representatives are rank-one projections on a Hilbert space) which preserve the quantity of transition probability (trace of the product of rank-one projections). Roughly speaking, Wigner's theorem states that any quantum mechanical symmetry transformation is implemented by either a unitary or an antiunitary operator on the underlying Hilbert space. The result in \cite{M} says that the same conclusion holds for the bijective transformations on the set of density operators which preserve the Umegaki relative entropy. 
In the paper \cite{MLSz} the bijectivity assumption was removed from the result in \cite{M} while in \cite{MN} the structures of preservers of other types of relative entropy were determined. After this, in \cite{MNSz} a far-reaching generalization of the previously mentioned results was given. Namely, all transformations on the set of density operators which preserve any so-called quantum $f$-divergence with respect to an arbitrary strictly convex function were determined.

Our present results are closely related to the aforementioned ones.
Here we consider the recently introduced and very extensively studied notion called quantum R\'enyi divergence \cite{Muletal13} (or, in another terminology, sandwiched R\' enyi relative entropy) and describe its preservers on the space of density operators as well as on the cone of all positive definite operators on a finite dimensional complex Hilbert space. Concerning some recent results on quantum R\'enyi divergence we refer e.g. to \cite{DaLed14, Dup15, FL, LinTom15, MosOga15, Wildeatal14}.

To formulate our present results we need a short summary of some notation, basic concepts and facts which is given in the next paragraphs.

Denote by $\mathbb R^+$ the set of all positive real numbers and set $\mathbb{R}^{+}_0= \mathbb R^+\cup \{ 0\}$.
Let $H$ be a finite dimensional complex Hilbert space. We denote by $B(H)$ the algebra of all linear operators on $H$, by $B(H)^+$ the cone of all positive semidefinite operators on $H$, and by $B(H)^{++}$ the cone of all positive definite (invertible positive semidefinite) operators on $H$. In what follows $S(H)$ stands for the set of all density operators on $H$ which are operators in $B(H)^+$ having unit trace. We recall that $B(H)$ is a complex Hilbert space with the Hilbert-Schmidt inner product $\langle.,.\rangle_{\HS}\colon B(H)\times B(H)\to\mathbb{C}$ defined by
\[
\langle A,B\rangle_{\HS}=\tr AB^*\quad(A,B\in B(H)), 
\]
$\tr$ standing for the usual trace functional on $B(H)$. 

Next we give the definition of quantum $f$-divergence. To do this,
for any $A\in B(H)$ we introduce the left and the right multiplication operators $L_A,R_A\colon B(H)\to B(H)$ defined by
\[
L_AT=AT,\ R_AT=TA\quad(T\in B(H)).
\]
Clearly, $L_AR_B=R_BL_A$ holds for every $A,B\in B(H)$. If $A,B\in B(H)^+$, then $L_A$ and $R_B$ are positive Hilbert space operators on $B(H)$ as a Hilbert space, hence so is their product $L_AR_B$.

Let $f\colon\mathbb{R}^{+}_0\to\R$ be a function which is
continuous on $\mathbb{R}^{+}$ and assume that the limit
\begin{equation*}\label{gamma}
\gamma=\lim_{t\to\infty}\frac{f(t)}{t}
\end{equation*}
exists in the extended real line $[-\infty,\infty]$.
Essentially following \cite[2.1 Definition]{HMPB}, for $A\in B(H)^+$ and $B\in B(H)^{++}$ the quantum $f$-divergence $S_f(A\Vert B)$ of $A$ and $B$ is defined by
\[
S_f(A\Vert B)=\left\langle\sqrt{B},f(L_AR_{B^{-1}})\sqrt{B}\right\rangle_{\HS},
\]
while in the general case, i.e. for $A,B\in B(H)^+$,  we set
\begin{equation*}\label{E:fformula}
S_f(A\Vert B)=\lim\limits_{\varepsilon \searrow 0}S_f(A\Vert B+\varepsilon I)
\end{equation*}
where $I$ stands for the identity operator on $H$. By \cite[2.2 Proposition]{HMPB} the limit above exists in $[-\infty,\infty]$ and it can be computed as follows. Let $A,B\in B(H)^+$ and for any $\l\in\R$ denote by $P_{\l}$, respectively by $Q_{\l}$ the projection on $H$ onto the kernel of $A-\l I$, respectively onto the kernel of $B-\l I$. According to \cite[2.3 Corollary]{HMPB} we have
\begin{equation}\label{E:formula}
S_f(A||B)=\sum_{a\in\sigma(A)}\left(\sum_{b\in\sigma(B)\backslash\{0\}}
bf\left(\frac{a}{b}\right)\tr P_aQ_b+\gamma a\tr P_aQ_0\right),
\end{equation}
where $\sigma(.)$ stands for the spectrum of operators in $B(H)$ and the convention $0\cdot(-\infty)=0\cdot\infty=0$ is used.

Two important examples of quantum $f$-divergences on density operators follow, see \cite[2.7 Example]{HMPB}.

\begin{itemize}
\item[(i)] If
\[
f(t)=\left\{
       \begin{array}{ll}
        t\log t, & t>0\\
        0, & t=0,
       \end{array}
     \right.
\]
then for $A,B\in \S$ we have
\[
S_f(A\Vert B)=\left\{
       \begin{array}{ll}
        \tr A(\log A-\log B), & \supp A\subset\supp B\\
        \infty, & \mbox{otherwise}
       \end{array}
     \right.
\]
which is just the usual Umegaki relative entropy of $A$ and $B$. Here and in what follows $\supp$ stands for the support of an operator which is the orthogonal complement of its kernel.
\item[(ii)] Let $\alpha \in \left]0,1\right[ \cup \left]1, \infty\right[ $ be fixed and $f(t)=t^\alpha$ $(t\geq 0)$. Pick $A,B\in \S$. For $\alpha\in \left]0,1\right[$ we have
\begin{equation*}\label{E:S1}
S_f(A\Vert B)=
        \tr A^{\alpha}B^{1-\alpha} 
\end{equation*}
and for $\alpha \in \left]1,\infty\right[$ we have
\begin{equation*}\label{E:S2}
S_f(A\Vert B)=\left\{
       \begin{array}{ll}
        \tr A^{\alpha}B^{1-\alpha}, & \supp A\subset\supp B\\
        \infty, & \mbox{otherwise.}
       \end{array}
     \right.
\end{equation*}
The "traditional" R\'enyi relative entropy $S_\alpha(.\Vert .)$ with parameter $\alpha$ is closely related to the quantum $f$-divergence above and is defined as follows. Pick $A,B\in \S$.
For $\alpha\in \left]0,1\right[$ set
\[
S_\alpha(A\Vert B)=\left\{
       \begin{array}{ll}
        (\alpha -1)^{-1}\log\left(\tr A^{\alpha}B^{1-\alpha}\right), & \supp A\not \perp \supp B \\
        \infty, & \mbox{otherwise.}
       \end{array}
     \right.
\]
and for $\alpha \in \left]1,\infty\right[$ set
\[
S_\alpha(A\Vert B)=\left\{
       \begin{array}{ll}
        (\alpha -1)^{-1}\log\left(\tr A^{\alpha}B^{1-\alpha}\right), & \supp A\subset\supp B \\
        \infty, & \mbox{otherwise.}
       \end{array}
     \right.
\]
\end{itemize}
We further mention that the quantum Tsallis relative entropy is a particular quantum $f$-divergence, see e.g. \cite{MNSz}, p. 2312.

As already written above, in \cite{MNSz} we determined all transformations (not necessarily bijective) of $\S$ which preserve the quantum $f$-divergence corresponding to any strictly convex (or strictly concave) function $f\colon\mathbb{R}^{+}_0\to\R$. We proved that any such transformation on $\S$ is necessarily implemented by a unitary or an antiunitary operator $H$ and hence the same conclusion holds for maps preserving R\'enyi relative entropy (and quantum Tsallis entropy) with any parameter. 

Let us now turn to the concept of the quantum R\'enyi divergence introduced in \cite{Muletal13}. We recall that it was also introduced in \cite{Wildeatal14} under the name "sandwiched R\'enyi relative entropy".
Referring to \cite{Muletal13}, the definition of the quantum R\'enyi divergence  $D_\alpha(.\Vert .)$ with parameter $\alpha$ is as follows. Pick arbitrary nonzero $A,B\in B(H)^+$. For $\alpha \in \left]0,1\right[$ define
\[
D_\alpha(A\Vert B)=\left\{
       \begin{array}{ll}
        (\alpha -1)^{-1}\log\left(\left(\tr A\right)^{-1}\tr \left(B^\frac{1-\alpha}{2\alpha}AB^\frac{1-\alpha}{2\alpha}\right)^{\alpha} \right), & \supp A\not \perp \supp B\\ 
        \infty, & \mbox{otherwise}       
       \end{array}
     \right.
\] 
and for $\alpha \in \left]1,\infty\right[$ define
\[
D_\alpha(A\Vert B)=\left\{
       \begin{array}{ll}
        (\alpha -1)^{-1}\log\left(\left(\tr A\right)^{-1}\tr \left(B^\frac{1-\alpha}{2\alpha}AB^\frac{1-\alpha}{2\alpha}\right)^{\alpha} \right), & \supp A\subset \supp B\\ 
        \infty, & \mbox{otherwise.}       
       \end{array}
     \right.
\]
Our primary aim in this paper is to describe the transformations of $\S$ which preserve the quantum R\'enyi divergence. Apparently, the problem is equivalent to the description of the preservers of the following related quantities. Pick any $A,B\in \S$ and
for $\alpha\in \left]0,1\right[$ set
\begin{equation}\label{E:dformula1}
D'_\alpha(A\Vert B)=
        \tr \left(B^\frac{1-\alpha}{2\alpha}AB^\frac{1-\alpha}{2\alpha}\right)^{\alpha}
\end{equation}
and for $\alpha \in \left]1,\infty\right[$ set
\begin{equation}\label{E:dformula2}
D'_\alpha(A\Vert B)=\left\{
       \begin{array}{ll}
        \tr \left(B^\frac{1-\alpha}{2\alpha}AB^\frac{1-\alpha}{2\alpha}\right)^{\alpha}, & \supp A\subset \supp B\\ 
        \infty, & \mbox{otherwise.}       
       \end{array}
     \right.
\end{equation}
Obviously, if these quantities were quantum $f$-divergences corresponding to some strictly convex (or strictly concave) functions, then our result in \cite{MNSz} would apply and we would be done. Therefore, we need to verify that the quantities $D'_\alpha$ are not quantum $f$-divergences. This is the content of our first proposition.


\begin{Proposition}\label{P:1}
For any $\alpha \in \left]0,1\right[ \cup \left]1, \infty\right[$ we have that $D'_\alpha(.||.)$ is not a quantum $f$-divergence on $\S$.
\end{Proposition}

Consequently, our original question does make sense and hence we can proceed. In fact, in what follows we solve a more general preserver problem. Namely, we introduce a quantity on $\S$ much more general than $D'_\alpha$ and describe the corresponding invariance transformations.

To do this, pick continuous functions $f:\mathbb{R}^+\to\mathbb{R}^+$ and $g:\mathbb{R}^{+}_0\to\mathbb{R}^{+}_0$. We must emphasize that this function $f$ has nothing to do with the function appearing in the concept of quantum $f$-divergence. We believe the use of the symbol $f$ in that other context causes no confusion. We define the quantity $D'_{f,g}(.\Vert .)$ for arbitrary $A\in B(H)^+$ and $B\in B(H)^{++}$ by
\[
D'_{f,g}(A\Vert B) = \tr g\left(f(B)Af(B)\right).
\]
Next, following the common approach used also in \eqref{E:fformula}, for any $A,B\in B(H)^+$ we would like to define 
\begin{equation}\label{E:1}
D'_{f,g}(A\Vert B) = \lim_{\varepsilon \searrow 0} D'_{f,g}(A\Vert B+\varepsilon I).
\end{equation}
In the next proposition we see that this can really be done in certain cases meaning that the limit in \eqref{E:1} exists under certain conditions on $f$ and $g$.

\begin{Proposition}\label{P}
Assume that $f:\mathbb{R}^+\to\mathbb{R}^+$ and $g:\mathbb{R}^{+}_0\to\mathbb{R}^{+}_0$ are continuous functions and $g(0)=0$. Select $A,B \in B(H)^+$ and denote by $P_B$ the orthogonal projection on $H$ onto the support of $B$.
\begin{itemize}
\item[(i)] If $\lim\limits_{\varepsilon \searrow 0}f(\varepsilon)=0$, then
the limit 
\begin{equation}\label{E:limit}
D'_{f,g}(A\Vert B)=\lim\limits_{\varepsilon \searrow 0} \tr g\left(f(B+\varepsilon I)Af(B+\varepsilon I)\right)
\end{equation}
exists and we have
\begin{equation} \label{dab1}
D'_{f,g}(A\Vert B) = \tr g\left(f(B|_{\supp B})P_BAP_Bf(B|_{\supp B})\right).
\end{equation}
(Here the operator $f(B|_{\supp B})P_BAP_Bf(B|_{\supp B})$ acts on $\supp B$.)
\item[(ii)] If $\lim\limits_{\varepsilon \searrow 0}f(\varepsilon)=\infty$, $g$ is monotone increasing and has limit $\infty$ at $\infty$, then
the limit \eqref{E:limit} exists and we have
\begin{equation} \label{dab2}
D'_{f,g}(A\Vert B)=\left\{
       \begin{array}{ll}
        \tr g\left(f(B|_{\supp B})P_BAP_Bf(B|_{\supp B})\right), & \supp A\subset\supp B\\
        \infty, & \mbox{otherwise.}
       \end{array}
     \right.
\end{equation}
\end{itemize}
\end{Proposition}

After this we can formulate the main results of the paper. First observe that any unitary or antiunitary similarity transformation on $S(H)$, i.e. any map on $S(H)$ of the form  $A\mapsto UAU^*$ with unitary or antiunitary operator $U$ on $H$, leaves the above defined quantities $D_{f,g}'(.\Vert .)$ invariant (see e.g. the first two sentences in the proof of Lemma~\ref{L:1}). In what follows we present results which state that
if $f$ and $g$ satisfy certain conditions then, conversely, any transformation $\phi$ (not necessarily bijective) on $\S$ that preserve $D'_{f,g}(.,\Vert .)$ is a unitary or antiunitary similarity transformation, i.e. induced by a unitary or an antiunitary operator on $H$. This means that the symmetries of $S(H)$ with respect to any member of a large class of "generalized" divergences are all the most simple transformations.


\begin{Theorem}\label{T:2}
Assume that $f:\mathbb{R}^+\to\mathbb{R}^+$ is a continuous function with $\lim\limits_{\varepsilon \searrow 0}f(\varepsilon)=0 $, and $g:\mathbb{R}^{+}_0\to\mathbb{R}^{+}_0$ is an injective continuous function with $g(0)=0$. If $\phi\colon \S\rightarrow \S$ is a transformation satisfying
\[
D'_{f,g}(\phi(A)\Vert\phi(B))=D'_{f,g}(A\Vert B)\quad (A,B\in \S)
\]
then there is either a unitary or an antiunitary operator $U$ on $H$ such
that $\phi$ is of the form
\[
 \phi(A)=UA U^*\qquad (A\in \S).
\]
\end{Theorem}

In the next result we obtain the same conclusion under different conditions on $f$ and $g$.


\begin{Theorem}\label{T:3}
Assume that $f:\mathbb{R}^+\to\mathbb{R}^+$ is a strictly monotone decreasing strictly convex function with $\lim\limits_{\varepsilon \searrow 0} f(\varepsilon)=+\infty $, and $g:\mathbb{R}^{+}_0\to\mathbb{R}^{+}_0$ is a strictly monotone increasing strictly convex (or strictly concave) continuous function with $g(0)=0$ and $\lim_{t\to \infty} g(t)=\infty$. If $\phi\colon \S\rightarrow \S$ is a transformation satisfying
\[
D'_{f,g}(\phi(A)\Vert\phi(B))=D'_{f,g}(A\Vert B)\quad(A,B\in \S)
\]
then there is either a unitary or an antiunitary operator $U$ on $H$ such
that $\phi$ is of the form
\[
 \phi(A)=UA U^*\qquad (A\in \S).
\]
\end{Theorem}

Clearly, the former statement implies that the transformations on $\S$ which preserve the quantum R\'enyi entropy with parameter $\alpha\in \left]0,1\right[$ are implemented by unitary or antiunitary operators. The latter statement does the same job for the case where $\alpha\in \left]1,\infty\right[$.
Therefore, we have the following immediate corollary.


\begin{Corollary}
If $\alpha \in \left]0,1\right[ \cup \left]1,\infty\right[$ and $\phi\colon \S\rightarrow \S$ is a transformation satisfying
\[
D_\alpha(\phi(A)||\phi(B))=D_\alpha(A||B)\quad (A,B\in \S)
\]
then there is either a unitary or an antiunitary operator $U$ on $H$ such
that $\phi$ is of the form
\[
 \phi(A)=UA U^*\qquad (A\in \S).
\]
\end{Corollary}

We emphasize that the bijectivity of the transformation $\phi$ is not assumed in the previous statements. However, we shall see that in the proofs we seriously use the fact that $\phi$ is a transformation mapping density operators to density operators. 
In certain investigations in quantum theory, especially where differential geometrical tools are applied, it is more natural to consider all positive (definite) operators and not only the densities normalized by the unit trace condition.
Concerning that setting we have the following result. We point out that below we consider transformations on the cone of all positive definite operators on $H$ but we need to assume their bijectivity. Moreover, we have the statement only for the quantum R\'enyi divergences, not for any more general numerical quantities. We note that, as we shall see, the proof of the following result is very much different from the proofs of Theorems~\ref{T:2} and \ref{T:3}. 


\begin{Theorem}\label{T:4}
Let $\alpha \in \left]0,1\right[ \cup \left]1,\infty\right[$. If $\phi\colon B(H)^{++} \rightarrow B(H)^{++}$ is a bijective transformation satisfying
\[
D_\alpha(\phi(A)||\phi(B))=D_\alpha(A||B)\quad (A,B\in B(H)^{++})
\]
then there is either a unitary or an antiunitary operator $U$ on $H$ and a scalar $c \in \mathbb{R}^+$ such
that $\phi$ is of the form
\[
 \phi(A)=cUA U^*\quad (A\in B(H)^{++}).
\]
\end{Theorem}

Of course, the natural question immediately arises that what happens if we omit the condition of bijectivity of the transformation above. We leave this as a probably highly nontrivial open problem.

\section{Proofs}

In this section we present the proofs of our results.
To the proof of Proposition~\ref{P:1} we need the following lemma.

\begin{Lemma}\label{L:fu}
Let $n\geq 2$ be an integer and $f:\mathbb{R}^{+}\to \mathbb R$ a function with the property that
\begin{equation}\label{E:fu}
\sum_{k=1}^n b_kf\left (\frac{a_k}{b_k}\right)=0
\end{equation}
holds whenever $a_1,\ldots, a_n$ and $b_1,\ldots, b_n$ are positive numbers such that $\sum_{k=1}^n a_k = \sum_{k=1}^n b_k=1$. Then there is a real number $c$ for which we have
$f(t)=c(t-1)$ $(t\in \mathbb{R}^{+})$.
\end{Lemma}

\begin{proof}
Choosing $a_k=b_k=1/n$ $(k=1,\ldots, n)$ we have $f(1)=0$. Let $t,s$ be positive real numbers less than, say, 0.9. Then, by the given property of $f$, we have
\begin{equation*}
tf\left(\frac{s}{t}\right)+
(0.9-t)f\left( \frac{0.9-s}{0.9-t}\right)+
\sum_{k=1}^{n-2}\frac{0.1}{n-2}f(1)=0.
\end{equation*}
(Observe that if $n=2$, then there is no need for the above "trick", instead of the number 0.9 we can choose 1 and the last sum in the above displayed equation does not show up.)
Since $f(1)=0$, it follows that
\begin{equation}\label{E:fu3}
tf\left( \frac{s}{t}\right)+
(0.9-t)f\left( \frac{0.9-s}{0.9-t}\right)=0
\end{equation}
for any real numbers $0<t,s<0.9$. Now fix $t,s$ such that $(0.9-s)/(0.9-t)=x$ be an arbitrarily preassigned positive number. Then for every real number $\lambda$ from a small enough neighborhood of $1$ we have
\begin{equation}\label{E:fu2}
\lambda t f\left( \frac{s}{t}\right) +(0.9-\lambda t)f\left( \frac{0.9-\lambda s}{0.9-\lambda t}\right)=0.
\end{equation}
The function 
\[
\lambda\mapsto \frac{0.9-\lambda s}{0.9-\lambda t}
\]
is strictly monotone and hence invertible with continuously differentiable inverse in a small neighborhood of $1$. From \eqref{E:fu2} we deduce that $f$ is continuously differentiable in a neighborhood of $x$. Since $x$ was an arbitrary positive number, it follows that $f$ is continuously differentiable on $\mathbb R^+$. Going back to \eqref{E:fu3} and differentiating with respect the variable $s$ we have
\begin{equation}\label{E:fu4}
f'\left( \frac{s}{t}\right)+
f'\left( \frac{0.9-s}{0.9-t}\right)=0
\end{equation}
for any $0<t,s<0.9$. Again, choosing particular $t,s$ such that $(0.9-s)/(0.9-t)=x$ is an arbitrarily preassigned positive number, and replacing $t,s$ by $\lambda t, \lambda s$ for $\lambda$ close enough to $1$, we see from \eqref{E:fu4} that $f'$ is constant in a neighborhood of $x$. Therefore, the continuous function $f'$ is locally constant which implies that it is globally constant. We deduce that $f$ is of the form $f(t)=ct+d$ $(t\in \mathbb R^+)$ and then by the property \eqref{E:fu} it follows easily that $d=-c$. This completes the proof of the lemma.
\end{proof}


Now we can present the proof of Proposition~\ref{P:1}.
Below we shall frequently use the following notation. For any vectors $x,y\in H$ we define the operator $x\otimes y\in B(H)$ by $(x\otimes y)z=\langle z,y\rangle x$ $(z\in H)$. It is apparent that $P\in B(H)$ is a rank-one projection if and only if there is a unit vector $x\in H$ such that $P=x\otimes x$. Elementary computation rules concerning the operation $\otimes$ are the following. For any $A\in B(H)$, $x,y\in H$ we have
\[
\begin{gathered}
A \cdot x\otimes y=(Ax)\otimes y\\
x\otimes y \cdot A=x\otimes (A^* y)\\
\tr (x\otimes y)=\langle x,y\rangle.
\end{gathered}
\]

\begin{proof}[Proof of Proposition~\ref{P:1}] 
Assume that for a given positive number $\alpha$ which is different from 1, $D'_\alpha$ is a quantum $f$-divergence, where $f:\mathbb R_0^+\to \R$ is function which is continuous on $\mathbb R^+$ and the limit $\gamma =\lim_{t\to \infty} f(t)/t$ exists in the extended sense. Let $n=\dim H$. For any positive numbers $a_1,\ldots, a_n$ and $b_1,\ldots, b_n$ with $\sum_{k=1}^n a_k = \sum_{k=1}^n b_k=1$ choose an orthonormal basis in $H$ and consider $A,B\in S(H)$ whose matrices with respect to that basis are $\Diag[a_1,\ldots, a_n]$ and $\Diag[b_1,\ldots, b_n]$, respectively. By $D'_\alpha (A\Vert B)=S_f(A\vert B)$ we have
\begin{equation*}
\sum_{k=1}^n b_k \left( \frac{a_k}{b_k}\right)^\alpha=
\sum_{k=1}^n b_k f\left( \frac{a_k}{b_k}\right).
\end{equation*}
Therefore, by Lemma \ref{L:fu} it follows that 
\begin{equation}\label{E:fu5}
f(t)=t^\alpha +c(t-1) \quad (t>0)
\end{equation}
holds for some real number $c$.
Substituting the matrices $\Diag[1,0,\ldots, 0]$ and $\Diag[b_1,\ldots, b_n]$ with positive diagonal entries into the formulas \eqref{E:formula} and \eqref{E:dformula1} or \eqref{E:dformula2}, we obtain
\[
b_1^{1-\alpha}=b_1f\left( \frac{1}{b_1}\right)+b_2f(0)+\ldots +b_n f(0)=b_1f\left( \frac{1}{b_1}\right)+(1-b_1)f(0).
\]
Taking the form in \eqref{E:fu5} into consideration we easily obtain $f(0)=-c$ and hence we have $f(t)=t^\alpha +c(t-1)$ for every $t\geq 0$.

Observe further, that in the case where $\alpha <1$ we have $\gamma=c$, while in the case where $\alpha>1$ we have $\gamma =\infty$. It is now not difficult to verify  (we omit the details) that for $\alpha <1$ we have $S_f(A\Vert B)=\tr A^\alpha B^{1-\alpha}$ $(A,B\in S(H))$ and for $\alpha >1$ we have
\[
S_f(A\Vert B)=\left\{
       \begin{array}{ll}
        \tr A^{\alpha}B^{1-\alpha}, & \supp A\subset\supp B \\
        \infty, & \mbox{otherwise}
       \end{array}
     \right.
\]
for any $A,B\in B(H)$. Cf. example (ii) in the introduction.

It follows that for any invertible density operator $B\in S(H)$ and arbitrary density operator $A\in S(H)$ we have 
\[
\tr A^\alpha B^{1-\alpha}=\tr \left(B^\frac{1-\alpha}{2\alpha}AB^\frac{1-\alpha}{2\alpha}\right)^{\alpha}.
\]
Substituting any rank-one projection $P=x\otimes x$ into the place of $A$ ($x\in H$ is an arbitrary unit vector), the above displayed equality implies that
\[
\langle B^{1-\alpha} x,x\rangle=\| B^\frac{1-\alpha}{2\alpha}x\|^{2\alpha}=\langle B^\frac{1-\alpha}{\alpha}x,x\rangle ^{\alpha}
\]
holds for any invertible density operator $B\in S(H)$ and unit vector $x\in H$. Considering the spectral decomposition $B=\sum_{b\in \sigma(B)} bQ_b$ of $B$, it follows that
\[
\sum_{b\in\sigma(B)} b^{1-\alpha} \langle Q_bx,x\rangle=
\left( \sum_{b\in\sigma(B)} b^\frac{1-\alpha}{\alpha} \langle Q_bx,x\rangle \right) ^{\alpha}.
\]
Here the only constraint regarding the numbers $\langle Q_bx,x\rangle$ $(b\in \sigma(B))$ is that they are non-negative and their sum is 1. In particular, for $0\leq t,s\leq 1/2$ we have
\[
\frac{t^{1-\alpha}+s^{1-\alpha}}{2}=\left( \frac{t^\frac{1-\alpha}{\alpha}+s^\frac{1-\alpha}{\alpha}}{2}\right)^{\alpha}
\]
which, by the strict convexity/concavity of the function $t\mapsto t^{1/\alpha}$ $(t\geq 0)$, leads to a contradiction. This proves our first proposition.
\end{proof}

Next we present two useful lemmas and then prove that the quantity $D'_{f,g}$ defined in \eqref{E:1} is well-defined when the functions $f,g$ satisfy certain conditions.

\begin{Lemma}\label{L:1}
Assume that $h: \mathbb{R}^+_0 \to \mathbb{R}^+_0$ is a continuous function. Then we have
\begin{equation*} \label{unieq}
\tr h(BAB) = \tr h\left(\sqrt{A}B^2\sqrt{A}\right) \quad (A,B\in B(H)^+).
\end{equation*}
\end{Lemma}

\begin{proof}
Observe that for any unitary operator $U\in B(H)$ and $X\in B(H)^+$ we have
$h(UXU^*)=Uh(X)U^*$ which follows from the fact that $h$ can be uniformly approximated by polynomials on compact sets. Hence $\tr h(UXU^*)= \tr h(X)$.
It is now sufficient to show that $BAB$ is unitarily similar to $\sqrt{A}B^2\sqrt{A}$, i.e. there is a unitary operator $U\in B(H)$ such that
\[
BAB = U^* \sqrt{A}B^2\sqrt{A} U.
\]
Considering the polar decomposition of $ \sqrt{A}B $ we have $ \sqrt{A}B = U |\sqrt{A}B | $, where $U$ is a partial isometry. Since $H$ is finite dimensional, any partial isometry can be extended to a unitary operator so we can assume that $U$ is unitary. Then we have
\[
\begin{gathered}
\sqrt{A}B^2\sqrt{A} =U|\sqrt{A}B | (U|\sqrt{A}B |)^* =
U|\sqrt{A}B |^2U^* = UBABU^*
\end{gathered}
\]
and we obtain our statement.
\end{proof}

In the next lemma we present a characterization of the order what we shall also need.

\begin{Lemma} \label{lemma:1}
Assume that $h: \mathbb{R}^+_0 \to \mathbb{R}^+_0$ is strictly monotone increasing continuous function with $h(0)=0$. Then for $B,C\in B(H)^+$ we have
\[
B^2 \leq C^2 \Longleftrightarrow  \tr h(BAB) \leq \tr h(CAC)\quad (A \in B(H)^{++}).
\]
\end{Lemma}

\begin{proof}
First we assume that $B^2 \leq C^2$ holds. Then for all $A \in B(H)^{++}$ we have $\sqrt{A}B^2\sqrt{A} \leq \sqrt{A}C^2\sqrt{A}$. The monotonicity of trace functions (see \cite[2.10. Theorem]{EC}) implies that
\[
\tr h\left(\sqrt{A}B^2\sqrt{A}\right) \leq \tr h\left(\sqrt{A}C^2\sqrt{A}\right).
\]
By Lemma~\ref{L:1} we deduce that for all $A \in B(H)^{++}$ the inequality $\tr h(BAB) \leq \tr h(CAC)$ is valid.

As for the converse statement, first observe that any rank-one projection can be approximated by a sequence of positive definite operators in the operator norm topology. So, assuming 
\[
\tr h(BAB) \leq \tr h(CAC)\quad (A \in B(H)^{++}),
\]
by the continuity of the function $h$ we obtain that
\[
\tr h(BPB) \leq \tr h(CPC)
\]
holds for every rank-one projection $P$ on $H$. Choosing any unit vector $x \in H$ and considering $P=x\otimes x$ we easily get  
\[
h\left(\|Bx\|^2\right) \leq h\left(\|Cx\|^2\right).
\]
By the strict monotonicity of $h$ we infer $\|Bx\|^2 \leq \|Cx\|^2$ for  every unit vector $x\in H$ which implies $B^2\leq C^2$. This completes the proof of the lemma.
\end{proof}

We are now in a position to give the proof of our second proposition.

\begin{proof}[Proof of Proposition~\ref{P}]
In the proof we apply the main ideas of the proof of \cite[Lemma 13]{Muletal13}.
Pick any $B \in B(H)^+$. With respect to the orthogonal decomposition $H=\supp B\oplus (\supp B)^\perp$ we can write
\[
B=\left(
\begin{array}{ccc}
B_0 & 0\\
0 & 0\\
\end{array}
\right)
\]
where $B_0=B|_{\supp B}$.
We choose an arbitrary $A \in B(H)^+$. With respect to the same orthogonal decomposition we have 
\[
A=\left(
\begin{array}{ccc}
A_0 & C\\
C^* & A_1\\
\end{array}
\right)
\]
where $A_0, C$ and $A_1$ are appropriate operators.
Easy computation gives 
\begin{equation} \label{formula}
f(B+\varepsilon I)Af(B+\varepsilon I)=\left(
\begin{array}{ccc}
f(B_0+\varepsilon I)A_0 f(B_0+\varepsilon I) & f(\varepsilon)f(B_0+\varepsilon I)C\\
f(\varepsilon)C^ *f(B_0+\varepsilon I) & f^2(\varepsilon)A_1\\
\end{array}
\right).
\end{equation} 
In this displayed formula as well as below, $I$ denotes the identity operator not necessarily on $H$ but on an appropriate subspace of it.
If $\lim_{\varepsilon \searrow 0}f(\varepsilon)=0$, we deduce from \eqref{formula} that 
\[
\lim\limits_{\varepsilon \searrow 0} f(B+\varepsilon I)Af(B+\varepsilon I) =
\left(
\begin{array}{ccc}
f(B_0)A_0 f(B_0) & 0\\
0 & 0\\
\end{array}
\right).
\]
By the continuity of $g$ and the property $g(0)=0$ it follows that
\[
\lim\limits_{\varepsilon \searrow 0} g(f(B+\varepsilon I)Af(B+\varepsilon I)) =
\left(
\begin{array}{ccc}
g(f(B_0)A_0 f(B_0)) & 0\\
0 & 0\\
\end{array}
\right)
\]
and we easily obtain \eqref{dab1}. 

In the case where $\lim_{\varepsilon \searrow 0}f(\varepsilon)=\infty$, under the assumption $\supp A \subset \supp B$ we have $C=0, A_1=0$ and it follows that the limit \eqref{E:limit} exists and we have
\[
D'_{f,g}(A\Vert B)=\tr g\left(f(B|_{\supp B})P_BAP_Bf(B|_{\supp B})\right).
\]
Assume now that $\supp A \not\subset \supp B$. Then there exists a unit vector $v \in H$ such that $v \in \ker B $ and $v \not\in \ker A$. With respect to the decomposition $\supp B \oplus \ker B$ of $H$, the vector $v$ is of the form 
\[
v = \left(
\begin{matrix}{}
0 \\
z
\end{matrix} \right)
\]
and
\begin{equation} \label{condition}
Av=\left(
\begin{array}{ccc}
A_0 & C\\
C^* & A_1\\
\end{array}
\right)\left(
\begin{matrix}{}
0 \\
z
\end{matrix} \right) = \left(
\begin{matrix}{}
Cz \\
A_1z
\end{matrix} \right) \neq \left(
\begin{matrix}{}
0 \\
0
\end{matrix} \right)
\end{equation}
holds. 
We claim $A_1 z \neq 0$. Assume on the contrary that $A_1z = 0$. Since $A \in B(H)^+$, for arbitrary $w\in \supp B$ we have
\[
\begin{gathered}
0 \leq \left\langle\left(
\begin{array}{ccc}
A_0 & C\\
C^* & A_1\\
\end{array}
\right)\left(
\begin{matrix}{}
w \\
z
\end{matrix} \right) , \left(
\begin{matrix}{}
w \\
z
\end{matrix} \right) \right\rangle =  \left\langle A_0 w,w\right\rangle + 2\Re\left\langle Cz,w\right\rangle.
\end{gathered}
\]
Hence for all $t \in \mathbb{R}$ and for an arbitrary $w\in \supp B$ we have
\[
0 \leq t^2 \left\langle A_0 w,w\right\rangle + 2t\Re\left\langle Cz,w\right\rangle
\]
which implies that for every $w\in \supp B$ the equality
\[
2\Re\left\langle Cz,w\right\rangle = 0
\]
holds. From this we deduce $Cz=0$ which contradicts \eqref{condition}. Therefore, we have $A_1z\neq 0$.

Denote by $Q$ the projection onto the subspace spanned by $v$,  i.e. let $Q=v\otimes v$. Recall $v\in \ker B$. We compute
\[
QAQ=\langle Av,v\rangle Q=\langle A_1 z,z\rangle Q.
\]
On the other hand, we have
\[
f(B+\varepsilon I)^2=
\left(
\begin{array}{ccc}
f(B_0+\varepsilon I)^2 & 0\\
0 & f^2(\varepsilon)I\\
\end{array}
\right)
\]
Since $Q$ projects onto a subspace of $\ker B$, it then follows that
\[
f(B+\varepsilon I)^2\geq f(\varepsilon)^2 Q.
\]
Therefore, applying Lemma \ref{lemma:1} we deduce
\[
\tr g(f(B+\varepsilon I)Af(B+\varepsilon I))\geq \tr g (f(\varepsilon)^2 QAQ)=g\left(f^2(\varepsilon)\left\langle A_1 z, z \right\rangle \right).
\]
Observe that the kernel of $A_1$ is the same as the kernel of its square root which implies that $\langle A_1 z, z \rangle$ is a positive real number.
By the properties of $f,g$ we see that in the latter displayed formula the right hand side quantity tends to infinity as $\varepsilon$ tends to zero. This completes the proof of the proposition.
\end{proof}

We can now turn to the proofs of our main results. Observe that by Proposition~\ref{P}, the quantity $D'_{f,g}$ in Theorems~\ref{T:2},\ref{T:3} is well-defined. We denote by $P_1(H)$ the set of all rank-one projections on $H$.


\begin{proof}[Proof of Theorem~\ref{T:2}] 
Assume the conditions in the statement hold and 
$\phi\colon \S\rightarrow \S$ is a transformation satisfying
\[
D'_{f,g}(\phi(A)\Vert\phi(B))=D'_{f,g}(A\Vert B)\quad (A,B\in \S).
\]
First we show that $\phi$ preserves the orthogonality in both directions,
i.e. it satisfies
\[
 \phi(A)\phi(B)=0 \Longleftrightarrow AB=0
\]
for any $A,B\in\S$. To see this we need the following characterization of
orthogonality. By the formula \eqref{dab1} and by the properties of $f,g$ it easily follows that for any $A,B \in \S$ we have
\begin{equation*}\label{ort}
AB=0\ \Longleftrightarrow\ D'_{f,g}(A\Vert B)=0.
\end{equation*}
Since $\phi$ preserves the quantity $D'_{f,g}(.\Vert .)$, it then follows that $\phi$ preserves the orthogonality in both directions.

Apparently, we can characterize the elements of $P_1(H)$ as those operators in $\S$ which belong to
a set of $n$ pairwise orthogonal density operators on $H$. By the orthogonality preserving property of $\phi$ we infer that it maps $P_1(H)$ into itself.
We claim that $\phi$ preserves also the transition probability (the trace of product) on $P_1(H)$. To prove this, let $P, Q\in P_1(H)$ be arbitrary. Applying \eqref{dab1} one can check that 
\[
D'_{f,g}(P\Vert Q)=g\left(f^2(1)\tr PQ\right) 
\]
and similarly
\[
D'_{f,g}(\phi(P)\Vert\phi(Q))=g\left(f^2(1)\tr \phi(P)\phi(Q)\right).
\]
By the injectivity of $g$ it follows that
\[
\tr \phi(P)\phi(Q)=\tr PQ.
\]
This means that the restriction of $\phi$ to $P_1(H)$ preserves the transition probability. The non-bijective version of Wigner's theorem (see e.g. \cite[Theorem 2.1.4]{MB}) describes the structure of all such maps. Since $H$ is finite dimensional, we obtain that there exists either a unitary or an antiunitary operator $U$ on $H$ such that
\[
\phi(P)=UPU^* \quad (P\in P_1(H)).
\]

Consider the transformation $\psi\colon \S\to\S$ defined by $\psi(A)=U^*\phi(A)U$ $(A\in \S)$.
It is clear that this map preserves the quantity $D'_{f,g}(A||B)$ and has the additional
property that it acts as the identity on $P_1(H)$.
Let $A\in\S$ be fixed and $Q\in P_1(H)$ be arbitrary. Using \eqref{dab1} again, we infer
\[
D'_{f,g}(A\Vert Q)=\tr g\left(f^2(1)QAQ\right)
\]
and similarly
\[
D'_{f,g}(\psi(A)\Vert Q)=\tr g\left(f^2(1)Q\psi(A)Q\right).
\]
By the properties of $\psi$ we have
\[
\tr g\left(f^2(1)QAQ\right)=\tr g\left(f^2(1)Q\psi(A)Q\right)
\]
holds for every rank-one projection $Q$ on $H$. Therefore, for every $x \in H$ with $\|x\|=1$ we deduce
\[
g\left(f^2(1)\left\langle Ax,x\right\rangle\right)=g\left(f^2(1)\left\langle \psi(A)x,x\right\rangle\right).
\]
Since $g$ is injective, it follows that
\[
\langle Ax,x\rangle = \langle \psi(A)x,x\rangle
\]
holds for every unit vector $x\in H$ and then we obtain
\[
A=\psi(A)=U^*\phi(A)U \quad (A\in\S).
\]
This completes the proof of the theorem.
\end{proof}

We next present the proof of our second main result.


\begin{proof}[Proof of Theorem~\ref{T:3}] 
The basic ideas of the argument below are close to those of the proof of \cite[Theorem]{MNSz} but there are smaller or bigger differences at many places. Therefore, for the sake of understandability, readability and completeness we present the proof with essentially all details.

Assume the conditions in the statement hold and 
$\phi\colon \S\rightarrow \S$ is a transformation satisfying
\[
D'_{f,g}(\phi(A)\Vert\phi(B))=D'_{f,g}(A\Vert B)\quad (A,B\in \S).
\]
We first show that $\phi$ preserves the rank, i.e. for any $A\in \S$ the rank of $\phi(A)$ equals the rank of $A$. In order to see it, let $A,B\in \S$ be arbitrary. Using \eqref{dab2}, it follows that $D'_{f,g}(A\Vert B)<\infty$ holds if and only if $\supp A\subset \supp B$.
We infer from this that
\[
\supp \phi(A)\subset \supp \phi(B) \Longleftrightarrow \supp A\subset \supp B
\]
next that
\[
\supp \phi(A)= \supp \phi(B) \Longleftrightarrow \supp A= \supp B
\]
and finally that
\begin{equation}\label{supp2}
\supp \phi(A) \subsetneq \supp \phi(B) \Longleftrightarrow \supp A\subsetneq \supp B.
\end{equation}
Observe that the rank of $A$ is $k$ if and only if there is a strictly increasing chain
(with respect to the relation of inclusion) of supports of $n$ density operators on $H$ such that its
$k$th element is $\supp A$. Using this characterization and \eqref{supp2} we see that $\phi$
leaves the rank of operators invariant. In particular, we have
\begin{equation}\label{E:pure}
\phi(P_1(H))\subset P_1(H).
\end{equation}

We next verify that $\phi$ is injective. Let $B,B'\in\S$ and suppose that $\phi(B) = \phi(B')$. For all $P \in P_1(H)$ we have $D'_{f,g}(\phi(P)\Vert \phi(B))=D'_{f,g}(\phi(P)\Vert \phi(B'))$ and by the preserver property of $\phi$ this implies $D'_{f,g}(P\Vert B)=D'_{f,g}(P\Vert B')$. Therefore, for any $P\in P_1(H)$ we have $D'_{f,g}(P\Vert B')< \infty $ if and only if $D'_{f,g}(P\Vert B) < \infty $ and hence we obtain $\supp B = \supp B'$.

Pick any $P\in P_1(H)$ with $\supp P \subset \supp B$ and apply \eqref{dab2} and the preserver property of $\phi$. We easily deduce that for every $x\in H$ with $\|x\|=1$ and $x \in \supp B$
\[
g\left(\|f(B|_{\supp B})x\|^2\right) = g\left(\|f(B'|_{\supp B'})x\|^2\right)
\]
holds. Due to the fact that $g$ is injective we conclude that 
\[
\langle f^2(B)x,x\rangle= \|f(B)x\|^2 = \|f(B')x\|^2= \langle f^2(B)x,x\rangle
\] 
and hence
$ f^2(B|_{\supp B}) = f^2(B'|_{\supp B'})$ is valid on $\supp B=\supp B'$. Since $f^2$ is strictly monotone decreasing we deduce $B=B'$ which proves that $\phi$ is injective.


In the next part of our argument we assume that $H$ is two-dimensional. We claim that for any $B\in\S$ we have
\[
[\min\sigma(B),\max\sigma(B)] \subset [\min\sigma(\phi(B)),\max\sigma(\phi(B))]
\]
meaning that $\phi$ can only enlarge the convex hull of the spectrum of the elements of $\S$.
To verify this property first observe that by \eqref{E:pure} the inclusion above holds for all $B\in P_1(H)$. Now pick a rank-two operator $B\in\S$ and set $\l=\max\sigma(B)\in[1/2,1[$. Then there are mutually orthogonal projections $P,Q\in P_1(H)$ such that $B=\l P+(1-\l)Q$.
Applying \eqref{dab2}, for any $R\in P_1(H)$ we obtain rather easily that
\begin{equation}\label{forma}
D'_{f,g}(R\Vert B)=g \left(f^2(\l)\tr PR+f^2(1-\l)\tr QR\right).
\end{equation}
Since $f$ is strictly monotone decreasing and $g$ is strictly monotone increasing, so  $g \circ f^2$ is strictly monotone decreasing on $\mathbb{R}^+$ and thus $g\left(f^2(\l)\right)\le g\left(f^2(1-\l)\right)$. It follows that as $R$ runs through the set $P_1(H)$, the numbers $\tr PR, \tr QR$ provide all pairs of non-negative reals with sum 1, and hence, using the continuity of $g$, the quantity $D'_{f,g}(R\Vert B)$ runs through the interval $[g\left(f^2(\l)\right),g\left(f^2(1-\l)\right)]$. Similarly, we infer that for any $R\in P_1(H)$ the number $D'_{f,g}(\phi(R)\Vert\phi(B))$ belongs to $[g\left(f^2(\mu)\right),g\left(f^2(1-\mu)\right)]$, where $\mu=\max\sigma(\phi(B))$. By the preserver property of $\phi$ we obtain that
\[
g\left(f^2(\mu)\right)\le g\left(f^2(\l)\right)\le g\left(f^2(1-\l)\right)\le g\left(f^2(1-\mu)\right).
\]
Due to the fact that $g\circ f^2$ is strictly monotone decreasing on $\mathbb{R}^+$ this implies
\[
\min\sigma(\phi(B))\le\min\sigma(B)\le\max\sigma(B)\le\max\sigma(\phi(B))
\]
which verifies our claim.


In the most crucial part of the proof which follows we show that $\phi\left(I/2\right)=I/2$. Assume on the contrary that there is a number $\l_1\in]1/2,1[$ and mutually orthogonal projections $P_1,Q_1\in P_1(H)$ for which
\begin{equation}\label{E:it1}
\phi\left(\frac{1}{2}I\right)=\l_1P_1+(1-\l_1)Q_1.
\end{equation}
By \eqref{forma} for any $R\in P_1(H)$ we have  $D'_{f,g}\left(R\left\Vert I/2\right.\right)=g(f^2\left(1/2\right))$ and then we deduce that
\begin{equation*}
\begin{gathered}
g\left(f^2\left(\frac{1}{2}\right)\right)=D'_{f,g}\left(\phi(R)\left\Vert\phi\left(\frac{1}{2}I\right)\right.\right)= \\
g\left(f^2(\l_1)\tr P_1\phi(R)
+f^2(1-\l_1)\tr Q_1\phi(R)\right).
\end{gathered}
\end{equation*}
Since $g$ is injective, we have
\begin{equation}\label{E:a}
f^2\left(\frac{1}{2}\right)=f^2(\l_1)\tr P_1\phi(R)
+f^2(1-\l_1)\tr Q_1\phi(R).
\end{equation}
As $1=\tr P_1\phi(R)+\tr Q_1\phi(R)$ holds, this gives us that $f^2\left(1/2\right)$ is a convex combination of $f^2(\l_1)$ and $f^2(1-\l_1)$. Since these latter numbers are different ($f^2$ is strictly monotone decreasing), we infer that $\tr P_1\phi(R)$ has the same value for any $R\in P_1(H)$ and the same holds for $\tr Q_{1}\phi(R)$, too. We next prove that
\begin{equation}\label{E:valami}
\tr P_1\phi(R)>\tr Q_1\phi(R).
\end{equation}
Due to the
strict convexity of $f$ we obtain $f^2$ is also strictly convex. Using that property and the fact that $f^2$ is strictly monotone decreasing, referring to \eqref{E:a} one can verify that $\tr P_1\phi(R)>1/2$ and then obtain $\tr P_1\phi(R)>\tr Q_1\phi(R)$. Indeed, in any representation of $f^2\left(1/2\right)$ as a convex combination of $f^2(t)$ and $f^2(1-t)\ (t\in]1/2,1[)$, the coefficient of the former term is necessarily greater than the coefficient of the latter one.

It follows from what we have observed above that when $R$ runs through the set $P_1(H)$, the number $\vartheta=\tr P_1\phi(R)$  remains constant, and since $f^2$ is clearly injective, $\vartheta$ is different from the numbers $0,1$. By \eqref{E:a} we have
\begin{equation}\label{teta}
\vartheta f^2(\l_1)+(1-\vartheta)f^2(1-\l_1)=f^2\left(\frac{1}{2}\right).
\end{equation}

Next let us consider $\phi\left(\phi\left(I/2\right)\right)$. We have
\[
\phi\left(\phi\left(\frac{1}{2}I\right)\right)=\lambda_2 P_2+(1-\l_2)Q_2
\]
for some $1/2\le\l_2<1$ and mutually orthogonal projections $P_2,Q_2$ in $P_1(H)$. In fact, as $\phi$ can only enlarge the convex hull of the spectrum and $\lambda_1>1/2$, it follows that $\lambda_2\geq \lambda_1>1/2$. Pick an arbitrary rank-one projection $R$ on $H$ and set $R_2=\phi(\phi(R))$. Since
$\phi$ preserves $D'_{f,g}(.\Vert.)$, similarly to \eqref{E:a} 
we have
\begin{equation*}
\begin{gathered}
g\left(f^2\left(\frac{1}{2}\right)\right)=
D'_{f,g}\left(\phi(\phi(R)) \left\Vert\phi\left(\phi\left(\frac{1}{2}I\right)\right.\right)\right) = \\
D'_{f,g}(R_2\Vert\lambda_2P_2+(1-\l_2)Q_2) = \\
g\left(f^2(\l_2)\tr P_2R_2+f^2(1-\l_2)\tr Q_2R_2\right).
\end{gathered}
\end{equation*}
This gives us that
\begin{equation}\label{E:5}
f^2\left(\frac{1}{2}\right)=
f^2(\l_2)\tr P_2R_2+f^2(1-\l_2)\tr Q_2R_2.
\end{equation}
Here $\lambda_2>1/2$ is fixed. Since the pair $\tr P_2R_2,\tr Q_2R_2$ of non-negative real numbers has sum 1, it follows just as above that the numbers $\tr P_2R_2$ and $\tr Q_2R_2$ are also fixed, they do not change when $R$ varies. Moreover, by the strict convexity of $f^2$ we also necessarily have
\begin{equation}\label{E:9}
\tr P_2R_2>\tr Q_2R_2.
\end{equation}

Now choose unit vectors $e$ and $f$ from the ranges of $P_1$ and $Q_1$, respectively. Consider a unit vector from the range of $P_2$. Let $\xi,\eta$ be its coordinates with respect to the basis $\{e,f\}$. It is easy to see that the representing matrix of $P_2$ is
\[
\left(
\begin{matrix}{}
\xi \\
\eta
\end{matrix} \right)
\left(
\begin{matrix}{}
\overline{\xi} \\
\overline{\eta}
\end{matrix} \right)^t,
\]
where $^t$ denotes the transposition. Moreover, since $R_2$ is a rank-one projection which is the image (under $\phi$) of a rank-one projection, its matrix representation is of the form
\[
\left(
\begin{matrix}{}
\vartheta & \varepsilon \sqrt{\vartheta(1-\vartheta)}\\
\overline{\varepsilon} \sqrt{\vartheta(1-\vartheta)} & 1-\vartheta
\end{matrix} \right),
\]
where $\vartheta$ is the same as in $\eqref{teta}$, and $\varepsilon\in\mathbb{C}$ with $|\varepsilon|=1$ varies as $R$ varies. We have
\[
\tr P_2R_2=
\tr \left[ \left(
\begin{matrix}{}
\xi  \\
\eta
\end{matrix} \right)
\left(
\begin{matrix}{}
\overline{\xi } \\
\overline{\eta}
\end{matrix} \right)^t
\left(
\begin{matrix}{}
\vartheta & \varepsilon\sqrt{\vartheta(1-\vartheta)}\\
\overline{\varepsilon}\sqrt{\vartheta(1-\vartheta)} & 1-\vartheta
\end{matrix} \right)
\right] .
\]
Elementary computations show that the latter quantity equals
\[
\begin{gathered}
\vartheta\xi\overline{\xi}+\sqrt{\vartheta(1-\vartheta)}\varepsilon\overline{\xi}\eta+\sqrt{\vartheta(1-\vartheta)}\overline{\varepsilon} \xi\overline{\eta}+(1-\vartheta)\eta\overline{\eta}=\\
\vartheta|\xi|^2+(1-\vartheta)|\eta|^2+2\sqrt{\vartheta(1-\vartheta)}\Re(\varepsilon\overline{\xi}\eta).
\end{gathered}
\]
As we have already showed, the value of $\tr P_2R_2$ does not change when $R$ varies and $\vartheta$ is also constant. Therefore, we obtain that the value of
\[
\vartheta|\xi|^2+(1-\vartheta)|\eta|^2+2\sqrt{\vartheta(1-\vartheta)}\Re (\varepsilon\overline{\xi}\eta)
\]
is the same for infinitely many values of $\varepsilon$ (by the injectivity of $\phi$ we see that $R_2$ runs through a set of continuum cardinality, so there is such a large set for the values of $\varepsilon$, too). It follows that $\Re(\varepsilon\overline{\xi}\eta)$ is the same for infinitely many values of $\varepsilon$ which clearly implies that $\overline{\xi}\eta=0$. Therefore, the column vector
\[
\left(
\begin{matrix}{}
\xi \\
\eta
\end{matrix} \right)
\]
is a scalar multiple of
\[
\left(
\begin{matrix}{}
1 \\
0
\end{matrix} \right)
\text{ or }
\left(
\begin{matrix}{}
0 \\
1
\end{matrix} \right).
\]
Obviously, this can happen only when $P_2=P_1$ or $P_2=Q_1$. Using the fact that $R_2$ is the image of a rank-one projection under $\phi$, it follows from \eqref{E:valami} that
\begin{equation}\label{E:8}
\tr P_1R_2>\tr Q_1R_2.
\end{equation}
If $P_2=Q_1$, then $P_1=Q_2$ and due to \eqref{E:9} we have
\[
\tr Q_1R_2>\tr P_1R_2
\]
which contradicts \eqref{E:8}. Therefore, the possibility $P_2=Q_1$ is ruled out and, consequently, we have $P_2=P_1$ and $Q_2=Q_1$. Thus we obtain
\begin{equation}\label{E:it2}
\phi\left(\phi\left(\frac{1}{2}I\right)\right)=\l_2P_1+(1-\l_2)Q_1.
\end{equation}
By \eqref{E:5} we have
\[
f^2(\lambda_2)\tr P_1R_2+f^2(1-\l_2)\tr Q_1R_2=f^2\left(\frac{1}{2}\right).
\]
On the other hand, referring to the sentence preceding \eqref{teta} we see that $\tr P_1R_2=\vartheta$ and $\tr Q_1R_2=1-\vartheta$, thus it follows that
\begin{equation}\label{E:l2}
\vartheta f^2(\l_2)+(1-\vartheta)f^2(1-\l_2)=f^2\left(\frac{1}{2}\right).
\end{equation}
We assert that the equation
\begin{equation}\label{E:f3}
\vartheta f^2(t)+(1-\vartheta)f^2(1-t)=f^2\left(\frac{1}{2}\right)
\end{equation}
has at most two solutions in $]0,1[$. Indeed, consider the function
\[
t\mapsto \vartheta f^2(t)+(1-\vartheta)f^2(1-t)\quad(t\in]0,1[).
\]
Since $f^2$ is strictly convex, the same holds for this function, too.
Therefore it cannot take the same values at three different places and 
hence \eqref{E:f3} does not have three different solutions in $]0,1[$. But by \eqref{teta} and \eqref{E:l2} $\l_1,\l_2$ and clearly $1/2$ too are solutions. Since $\l_2\ge\l_1>1/2$, it then follows that $\l_2=\l_1$ and referring to \eqref{E:it1} and \eqref{E:it2} we see that $\phi\left(\phi\left(I/2\right)\right)=\phi\left(I/2\right)$. Since $\phi$ is injective, this gives us that $\phi\left(I/2\right)=I/2$. Therefore, $\phi$ sends $I/2$ to itself.


Now let $I/2\ne A\in\S$ be a rank-two operator and denote by $\l\in]1/2,1[$ its maximal eigenvalue. We assert that $\sigma(\phi(A))=\sigma(A)$. Let $h\colon]0,1[\to\mathbb{R}$ be the function defined by
\[
h(t)=g\left(f^2\left(\frac{1}{2}\right)t\right)+g\left(f^2\left(\frac{1}{2}\right)(1-t)\right)\quad(t\in]0,1[).
\]
Using the formula \eqref{dab2} we obtain 
\[
D'_{f,g}\left(A\left\Vert\frac{1}{2}I\right.\right)=h(\l)
\]
and, similarly, 
\[
D'_{f,g}\left(\phi(A)\left\Vert\frac{1}{2}I\right.\right)=h(\l'),
\] 
where $\l'=\max\sigma(\phi(A))>1/2$. Since $\phi$ preserves $D'_{f,g}(.\Vert.)$ and sends $I/2$ to itself, it follows that $D'_{f,g}\left(\phi(A)\left\Vert I/2\right.\right)=D'_{f,g}\left(A\left\Vert I/2\right.\right)$, and hence that $h(\l)=h(\l')$. If $g$ is assumed to be strictly convex (the case when $g$ is strictly concave can be handled in a similar way), then we have that $h$ is strictly convex and symmetric with respect to the middle point $1/2$ of its domain. By elementary properties of convex functions this implies that the restriction of $h$ to $]1/2,1[$ is strictly monotone increasing. We necessarily obtain that $\l=\l'$ and this yields that the spectrum of $A$ coincides with that of $\phi(A)$. Therefore, $\phi$ is spectrum preserving.


Select mutually orthogonal projections $P,Q\in P_1(H)$ and pick a number $\l \in ]1/2,1[$. Consider the operator $B=\l P+(1-\l)Q$.
By the spectrum preserving property of $\phi$ we can choose another pair $P',Q'\in P_1(H)$ of mutually orthogonal projections such that $\phi(B)=\l P'+(1-\l)Q'$. We have learnt before (see the discussion around \eqref{forma}) that when $R$ runs through the set of all rank-one projections, the quantity $D'_{f,g}(R\Vert B)$ runs through the interval $[g\left(f^2(\l)\right),g\left(f^2(1-\l)\right)]$.
Using the equation \eqref{forma} we easily see that
$D'_{f,g}(R\Vert B)=g\left(f^2(\lambda)\right)$ if and only if $\tr PR=1$ which holds exactly when $R=P$.
Therefore, we obtain
\[
\begin{gathered}
R=P \Longleftrightarrow D'_{f,g}(R\Vert B)=g\left(f^2(\lambda)\right) 
\\ 
\Longleftrightarrow D'_{f,g}(\phi(R)\Vert \phi(B))= g\left(f^2(\lambda)\right)
\\
\Longleftrightarrow D'_{f,g}(\phi(R)\Vert \lambda P'+(1-\lambda) Q')
=g\left(f^2(\lambda)\right) 
\\
\Longleftrightarrow \phi(R)=P'.
\end{gathered}
\]
This gives us that $\phi(P)=P'$ and we similarly obtain $\phi(Q)=Q'$. Consequently, $\phi$ preserves the orthogonality between rank-one projections and we also have
\begin{equation}\label{E:17}
\phi(B)=\phi(\lambda P+(1-\l)Q)
=\lambda\phi(P)+(1-\l)\phi(Q).
\end{equation}


Next, we show that $\phi$ preserves the nonzero transition probabilities between rank-one projections.
Let $P$ and $R$ be different rank-one projections which are not orthogonal to each other.
Choose a rank-one projection $Q$ which is orthogonal to $P$.
Pick $\lambda \in ]1/2,1[$. On the one hand, we have
\[
D'_{f,g}(R\Vert \lambda P+(1-\lambda)Q)=g\left(f^2(\l)\tr PR  +f^2(1-\l)\tr QR \right)
\]
and on the other hand, by \eqref{E:17}, we compute
\[
\begin{gathered}
D'_{f,g}(R\Vert \lambda P+(1-\l) Q)= D'_{f,g}(\phi(R)\Vert \lambda \phi(P)+(1-\l) \phi(Q))\\=g\left(f^2 (\lambda) \tr \phi(P)\phi(R) +f^2(1-\l) \tr \phi(Q)\phi(R)\right).
\end{gathered}
\]
Comparing the right-hand sides and using the injectivity of $g$, we infer
\[
\tr PR=\tr \phi(P)\phi(R).
\]
Consequently, $\phi$ preserves the transition probability between rank-one projections.

Above we have supposed that $H$ is two-dimensional. Assume now that $H$ is an arbitrary finite dimensional Hilbert space and $\phi:S(H)\to S(H)$ is a transformation which preserves the quantity $D'_{f,g}(.\Vert.)$. We show that $\phi$ preserves the transition probability between rank-one projections in this case too. In fact, we can reduce the general case to the previous one. To see this, first let $H_2$ be a two-dimensional subspace of $H$ and $A_0\in S(H)$ be such that $\supp A_0=H_2$. Set $H_2'=\supp \phi(A_0)$. Since $\phi$ preserves the rank, $H_2'$ is also two-dimensional. By what we have learnt at the beginning of the proof, $\phi$ maps any element of $S(H)$ whose support is included in $H_2$ to an element of $S(H)$ whose support is included in $H_2'$. In that way $\phi$ gives rise to a transformation $\phi_0 :S(H_2) \to S(H_2')$ which preserves the quantity $D'_{f,g}(.\Vert.)$. Consider a unitary operator $V:H_2'\to H_2$. The transformation $V\phi_0(.)V^*$ maps $S(H_2)$ into itself and preserves the quantity $D'_{f,g}(.\Vert.)$. We have already seen that such a transformation necessarily preserves the transition probability between rank-one projections which implies that the same holds for $\phi_0$ as well. Since for any two rank-one projections $P,Q$ there exists a rank-two element $A_0\in S(H)$ such that $\supp P,\supp Q\subset \supp A_0$, it follows that we have
\[
\tr PQ=\tr \phi(P)\phi(Q).
\]

By the non-bijective version of Wigner's theorem we infer that there is either a unitary or an antiunitary operator $U$ on $H$ such that
\[
\phi(P)=UPU^*\quad(P\in P_1(H)).
\]
Define the map $\psi\colon\S\to\S$ by $\psi(A)=U^*\phi(A)U\ (A\in\S)$. It is clear that $\psi$ preserves $D'_{f,g}(.\Vert.)$ and it acts as the identity on $P_1(H)$. Let $A\in \S$. Since $\psi$ leaves the quantity $D'_{f,g}(.\Vert.)$ invariant, it preserves the
inclusion between the supports of elements of $\S$ (see the first part of the proof). This implies that for every rank-one projection
$P$ on $H$ we have
\[
\supp P \subset \supp A \Longleftrightarrow \supp P \subset \supp\psi(A).
\]
We easily obtain that $\supp A=\supp\psi(A)$. Let $P$ be an arbitrary rank-one projection which satisfies $\supp P \subset\supp A=\supp \psi(A)$.
Using \eqref{dab2} and the equality $D'_{f,g}(P\Vert \psi(A))=D'_{f,g}(P\Vert A)$ we deduce that for any $x \in \supp A$ with $\|x\|=1$ the equation
\[
 g\left(\|f(\psi(A)|_{\supp A})x\|^2\right)= g\left(\|f(A|_{\supp A})x\|^2\right).
\]
holds. Just as at the end of the proof of Theorem~\ref{T:2} it follows that $f^2(\psi(A)|_{\supp A})$ equals $f^2(A|_{\supp A})$.
Using the injectivity of $f^2$ we can infer that $\psi(A)=A$ and next that $\phi(A)=UAU^*$. This completes the proof of the theorem.
\end{proof}

Finally, we present the proof of our last result.

\begin{proof}[Proof of Theorem~\ref{T:4}] 
As a consequence of Lemma \ref{lemma:1}, by the preservation of $D_\alpha(.||.)$ under the transformation $\phi$ we infer that the following equivalences hold
\[
\begin{gathered}
B^\frac{1-\alpha}{\alpha} \leq C^\frac{1-\alpha}{\alpha}  \\\Longleftrightarrow 
\tr \left(B^\frac{1-\alpha}{2\alpha}AB^\frac{1-\alpha}{2\alpha}\right)^\alpha \leq 
\tr \left(C^\frac{1-\alpha}{2\alpha}AC^\frac{1-\alpha}{2\alpha}\right)^\alpha \quad (A \in B(H)^{++})\\ \Longleftrightarrow 
\tr \left(\phi(B)^\frac{1-\alpha}{2\alpha}\phi(A)\phi(B)^\frac{1-\alpha}{2\alpha}\right)^\alpha \leq 
\tr \left(\phi(C)^\frac{1-\alpha}{2\alpha}\phi(A)\phi(C)^\frac{1-\alpha}{2\alpha}\right)^\alpha (A \in B(H)^{++})\\ \Longleftrightarrow
\phi(B)^\frac{1-\alpha}{\alpha} \leq \phi(C)^\frac{1-\alpha}{\alpha}.
\end{gathered}
\]
This implies that
\[
B \leq C  \Longleftrightarrow  \phi(B^\frac{\alpha}{1-\alpha})^\frac{1-\alpha}{\alpha} \leq \phi(C^\frac{\alpha}{1-\alpha})^\frac{1-\alpha}{\alpha}
\]
is valid for any $B,C \in B(H)^{++}$.
We conclude that the bijective map $\psi\colon B(H)^{++} \rightarrow B(H)^{++}$ defined by 
\[
\psi(X)= \phi(X^\frac{\alpha}{1-\alpha})^\frac{1-\alpha}{\alpha} \quad (X\in B(H)^{++})
\] 
is an order automorphism of $B(H)^{++}$. The structure of such transformations is described in \cite{MOrd}. It follows from \cite[Theorem 1]{MOrd} that $\psi$ is of the form
\[
\psi(X)=TXT^*\quad (X\in B(H)^{++})
\]
where $T$ is an invertible linear or conjugate-linear operator on $H$. By the definition of $\psi$ we have
\begin{equation} \label{phiform}
\phi(X)=\left(TX^\frac{1-\alpha}{\alpha}T^*\right)^\frac{\alpha}{1-\alpha} \quad (X\in B(H)^{++}).
\end{equation}
Consider the polar decomposition $T=U| T|$ where $U$ is a unitary or antiunitary operator on $H$. We apparently have
\[
\phi(X)=\left(U|T|X^\frac{1-\alpha}{\alpha}|T|U^*\right)^\frac{\alpha}{1-\alpha}=U\left(|T|X^\frac{1-\alpha}{\alpha}|T|\right)^\frac{\alpha}{1-\alpha} U^*\quad (X\in B(H)^{++}).
\]
Since the unitary as well as antiunitary similarity transformations are clearly invariant under $D_\alpha(.\Vert .)$, without serious loss of generality we can and do assume that in \eqref{phiform} we have $T \in B(H)^{++} $. Our aim now is to show $T$ is a scalar multiple of the identity. Using the preserver property of $\phi$ and \eqref{phiform}, we deduce that
\begin{equation}\label{E:M2} 
\begin{gathered}
\frac{1}{\tr \left(TA^{\frac{1-\alpha}{\alpha}}T\right)^{\frac{\alpha}{1-\alpha}}}  
\tr \left(\left(TB^{\frac{1-\alpha}{\alpha}}T\right)^{\frac{1}{2}} 
\left(TA^{\frac{1-\alpha}{\alpha}}T\right)^{\frac{\alpha}{1-\alpha}}
\left(TB^{\frac{1-\alpha}{\alpha}}T\right)^{\frac{1}{2}}\right)^\alpha \\
= \frac{1}{\tr A} \tr \left(B^\frac{1-\alpha}{2\alpha}AB^\frac{1-\alpha}{2\alpha}\right)^\alpha
\end{gathered}
\end{equation}
holds for all $A,B\in B(H)^{++}$. 
Let $B=T^{\frac{-2\alpha}{1-\alpha}}$ and $A=I$. We obtain from \eqref{E:M2} that
\[
\frac{1}{\tr T^\frac{2\alpha}{1-\alpha}} 
\tr T^{\frac{2\alpha^2}{1-\alpha}}=\frac{1}{\tr I} \tr T^{-2\alpha}
\]
or, equivalently,
\[
(\tr I)  
(\tr T^{\frac{2\alpha^2}{1-\alpha}})= (\tr T^{-2\alpha}) ({\tr T^\frac{2\alpha}{1-\alpha}}).
\]
Let $t_1, \ldots, t_n$ be the eigenvalues of the positive invertible operator $T$ listed in decreasing order. By the last displayed formula, for the finite sequences $x_k=t_k^{-2\alpha}$, $y_k=t_k^{\frac{2\alpha}{1-\alpha}}$ $(k=1, \ldots, n)$ we have
\[
\frac{\sum_{k=1}^n x_k y_k}{n}=\frac{\sum_{k=1}^n x_k}{n} \frac{\sum_{k=1}^n y_k}{n}.
\]
Depending on $\alpha >1$ or $\alpha <1$, the finite sequences $x_1, \ldots, x_n$ and $y_1, \ldots, y_n$ are either similarly ordered or oppositely ordered. By Tchebychef's inequality (see e.g. 2.17. in \cite{HLP}) it follows that either the $x_k$'s or the $y_k$'s are equal. In either case we have the $t_k$'s are equal implying that $T$ is a positive constant multiple of the identity. This completes the proof of the theorem. 
\end{proof}

\bibliographystyle{amsplain}

\end{document}